\documentclass[a4]{amsart}

\usepackage{amssymb,amsmath,latexsym,amscd,amsfonts}
\usepackage{mathrsfs}
\usepackage{hyperref}
\usepackage{cleveref}

\usepackage[margin=3.4cm]{geometry}

\usepackage{graphicx}
\usepackage[all]{xy}
\usepackage{xypic}
\usepackage{psfrag}
\usepackage[svgnames]{xcolor}

\numberwithin{equation}{section}
\newtheorem{theorem}{Theorem}[section]

\newtheorem{proposition}[theorem]{Proposition}
\newtheorem{corollary}[theorem]{Corollary}
\newtheorem{lemma}[theorem]{Lemma}

\theoremstyle{definition}
\newtheorem{definition}[theorem]{Definition}
\newtheorem{notation}[theorem]{Notation}
\newtheorem{example}[theorem]{Example}
\newtheorem{remark*}[theorem]{}
\theoremstyle{remark}
\newtheorem{remark}[theorem]{Remark}

\newcommand{\rar}{\rightarrow}

\newcommand{\C}{\mathbb C}
\newcommand{\Z}{\mathbb Z}
\newcommand{\N}{\mathbb N}

\newcommand{\tr}{\mathrm{tr}}

\def\la{\langle}
\def\ra{\rangle}

\begin{document}
\title{A note on irreducible quadrilaterals of
  $II_1$ factors}

\author[K C Bakshi]{Keshab Chandra Bakshi}
\address{Chennai Mathematical Institute, Chennai, INDIA}
\email{bakshi209@gmail.com, kcbakshi@cmi.ac.in}

\author[V P Gupta]{Ved Prakash Gupta}
\address{School of Physical Sciences, Jawaharlal Nehru University, New Delhi, INDIA}
\email{ved.math@gmail.com, vedgupta@mail.jnu.ac.in}
\thanks{The first named author was supported  by a
  postdoctoral fellowship of the National Board for Higher Mathematics (NBHM),
  India.}

\begin{abstract}
 Given any quadruple $(N, P, Q, M)$ of $II_1$-factors with finite index,
 the notions of interior and exterior angles between $P$ and $Q$ were
 introduced in \cite{BDLR2017}. We determine the possible values of
 these angles in terms of the cardinalities of the Weyl groups of the
 intermediate subfactors when $(N, P, Q, M)$ is an irreducible
 quadrilateral and the subfactors $N \subset P$ and $N \subset Q$ are
 both regular. For an arbitrary irreducible quadruple, an attempt is
 made to determine the values of angles by deriving expressions for
 the angles in terms of the common norm of two naturally arising
 auxiliary operators and the indices of the intermediate subfactors of
 the quadruple. Finally, certain bounds on angles between $P$ and $Q$
 are obtained when $N \subset P$ is regular, which enforce some
 restrictions on the index of $N \subset Q$ in terms of that of $N
 \subset P$.

\end{abstract}
\maketitle
\section{Introduction} 
A quadrilateral is a quadruple $(N, P, Q, M)$ of $II_1$-factors such
that $N \subset P, Q \subset M$, $N = P \wedge Q$, $M = P \vee Q$ and
$[M:N] <\infty$; it is called irreducible if $N \subset M$
is irreducible.  The Weyl group of a finite index $II_1$-subfactor $N
\subset M$ is the quotient group $G:=\mathcal{N}_M(N)/\mathcal{U}(N)$,
where $\mathcal{N}_M(N)$ denotes the group of unitary normalizers of
$N$ in $M$, i.e., $\mathcal{N}_M(N):=\{u \in \mathcal{U}(M) : u N u^*
= N\}$. This article concentrates mainly on the analysis of such
quadrilaterals from the perspectives of (a) calculating the interior
and exterior angles between $P$ and $Q$ as was introduced in
\cite{BDLR2017}, (b) understanding the Weyl group of $N \subset M$ in
terms of those of $N \subset P$ and $N \subset Q$, and (c)
establishing a relationship between the above two aspects.

 Unlike the notion of set of angles by Sano and Watatani (\cite{SW}),
 the interior and exterior angles are both single entities and are
 seemingly more calculable, as we show in \Cref{examples} by making
 some explicit calculations.  As an important application of the
 notion of interior angle, the authors in \cite{BDLR2017} were able to
 improve a result of Longo \cite{Lon} by providing a better bound for
 the number of intermediate subfactors of a given irreducible subfactor.

 A natural question that struck us, after the appearance of
 \cite{BDLR2017}, was to determine the possible set of values that
 the interior and exterior angles can attain.  This article is devoted to
 this theme.  In general, it looks like a tough nut to crack. However,
 in the irreducible set up, we see that these angles take some definitive
 values.

 In \Cref{angles}, we discuss various generalities and formulae
 related to the interior and  exterior angles and employ them to compute
 angles between two intermediate subfators associated with a quadruple
 of crossed product algebras.

 In \Cref{quadrilaterals-regularity}, our main focus is on irreducible
 quadrilaterals $(N, P, Q, M)$ for which $N\subset P$ and $N\subset Q$
 are both regular. Recall that an unital inclusion of von Neumann
 algebras $A \subset B$ is said to be regular if $\mathcal{N}_B(A)'' =
 B$. Jones, in \cite{Jon}, had asked whether an irreducible regular
 subfactor is always a group subfactor. Making use of a theorem of
 Sutherland \cite{Sut} on vanishing of cohomologies, Popa \cite{PiPo2}
 and Kosaki \cite{Kos} (for properly infinite case) answered Jones'
 question in the affirmative, which was announced earlier for the
 hyperfinite case by Ocnenanu in 1986. Later, Hong gave an explicit
 realization of the same in \cite{Hong}. Using Hong's technique, we
 deduce (in \Cref{hk-generate-g}) that an irreducible quadrilateral
 $(N, P, Q, M)$ with regular $N \subset P$ and $N \subset Q$ can be
 realized as a quadrilateral of crossed product algebras through outer
 actions of Weyl groups. Using this realization and the calculations
 of \Cref{examples}, we provide a direct relationship between the
 interior and exterior angles between $P$ and $Q$ and the Weyl groups
 of $N \subset P$ and $N \subset Q$ in:\smallskip
 
\noindent{\bf \Cref{betacomutationforregularquadrilateral}} {\em Let
  $(N,P,Q,M)$ be an irreducible quadrilateral such that $N\subset P$
  and $N\subset Q$ are both regular. Then, $\alpha(P,Q)=\pi/2$,
  i.e., $(N, P, Q, M)$ is a commuting square, and
 $$\cos \big(\beta(P,Q)\big)=\displaystyle \frac{\frac{|G|}{\lvert H\rvert
      \lvert K\rvert}-1}{\sqrt{[G:H]-1}\sqrt{[G:K]-1}},$$
  where $H, K$ and $G$ denote the Weyl groups of $N \subset P$, $N
  \subset Q$ and $N \subset M$, respectively.

  In particular, $(N, P, Q, M)$ is a cocommuting square if and only if
  $G = HK$.  }\smallskip

\Cref{values} dwells around the main theme of this article, viz., to
determine the possible values of the interior and exterior angles. We
first derive expressions for the angles in terms of the common norm
$\lambda$ of two naturally arising auxiliary operators and the indices
of the intermediate subfactors of the quadruple (in
\Cref{formulaalphabeta}).  Then, in the irreducible setup, we exploit
these expressions to obtain some definitive values for angles by
making use of above relationship between angles and Weyl groups, a
theorem of Popa \cite{pop} wherein he determines the possible values
taken by the set $\Lambda(M, N)$ of relative dimensions of
projections, and relating $\lambda$ with certain polynomials $P_n(x),
n \geq 0$, which are near relatives of the Chebyshev polynomials as
introduced by Jones in \cite{Jon}. The results that we prove
are:\smallskip

\noindent{\bf \Cref{valuesofangle}} {\em Let $(N,P,Q,M)$ be a
  quadruple with $N\subset M$ irreducible, $r:= \frac{[Q:N]}{[M:P]}$, $\tau:=[M:N]^{-1}$ and let $0<t\leq 1/2$ be such that $t(1-t)=\tau$. If
  $r/\lambda \geq t$, then,
$$\cos(\alpha(P,Q)) \leq
\frac{[P:N][Q:N](1-t)-1}{\sqrt{[P:N]-1}\sqrt{[Q:N]-1}}$$ and $$
\cos(\beta(P,Q))\leq \displaystyle
\frac{\frac{1}{t}-1}{\sqrt{[M:P]-1}\sqrt{[M:Q]-1}}.$$
And, if $r/\lambda < t$, then, 
$$ \cos(\alpha(P,Q)) =
\displaystyle\frac{[P:N][Q:N]\frac{P_k(\tau)}{P_{k-1}(\tau)}-1}{\sqrt{[P:N]-1}\sqrt{[Q:N]-1}}
$$
and
$$\cos(\beta(P,Q))=\displaystyle \frac{\frac{P_k(\tau)}{\tau
    P_{k-1}(\tau)}-1}{\sqrt{[M:P]-1}\sqrt{[M:Q]-1}}$$ for some $k\geq
0$.}
\smallskip

\noindent{\bf \Cref{valuesofbeta}} {\em  Let $(N,P,Q,M)$ be an irreducible
  quadrilateral such that $N\subset P$ and $N\subset Q$ are
both  regular and suppose $[P:N]=2$.  Then, $\cos(\beta(P,Q))=
  \displaystyle \frac{P_2(m/2)}{\sqrt{P_2(\delta^2/2)P_3(m/2)}}$,
where $m = [M:Q] \in\N$ and, as usual $\delta:=\sqrt{[M:N]}.$}\smallskip

\noindent As a `geometric' consequence, in \Cref{valuesofbeta2}, we
see that if both $N \subset P$ and $N \subset Q$ have index $2$, then
the exterior angle $\beta(P,Q) > \pi/3$.\smallskip

Finally, while analyzing a quadrilateral intuitively as a picture in the plane
(\Cref{fig-2}), loosely speaking, we realize in \Cref{angle-bounds}
that the angles impose some sort of rigidity on the lengths of its
sides. This could be inferred as a direct consequence of certain bounds on
interior and exterior angles that we obtain in:\smallskip

\noindent{\bf \Cref{inequality}}
{\em Let $(N,P,Q,M)$ be a finite index irreducible quadruple such that
 $N\subset P$ is regular. Then,
 $$\cos\big(\alpha(P,Q)\big)\leq
  \Bigg(\sqrt{\frac{[P:N]-1}{[Q:N]-1}}\Bigg)$$ and
  $$\cos(\beta(P,Q))\leq \Bigg(\sqrt{\frac{[P:N]-r}{[Q:N]-r}}\Bigg).$$
}\smallskip

The flow of the article revolves around the results mentioned above,
more or less in the same order.

\section{Interior and Exterior angles between intermediate subfactors}\label{angle-generalities}\label{angles}
In this section, we first recall the notions of interior and exterior
angles between intermediate subfactors of a given subfactor as
introducted by Bakshi {\em et al.} in \cite{BDLR2017} and some useful
formulae related to them. This will be followed by some further
generalities and explicit calculations related to these angles.

In this article, we will be dealing only with subfactors and quadruples of type
$II_1$ with finite Jones' index.  Given any such quadruple $$\begin{matrix}
  Q &\subset & M \cr \cup &\ &\cup\cr N &\subset & P,
\end{matrix}$$  consider the basic
 constructions $N\subset M \subset M_1$, $P \subset M \subset P_1$ and
 $Q \subset M \subset Q_1$. As is standard, we denote by $e_1$ the
 Jones projection $e^M_N$. It is easily seen that, as $II_1$-factors
 acting on $L^2(M)$, both $ P_1$ and $Q_1$ are contained in $ M_1$. In
 particular, if $e_P: L^2(M) \rar L^2(P)$ denotes the orthogonal
 projection, then $e_P \in M_1$. Likewise, $e_Q \in M_1$. Thus, we
 naturally obtain a dual quadruple $$\begin{matrix} P_1 &\subset & M_1
   \cr \cup &\ &\cup\cr M &\subset & Q_1.
\end{matrix}$$

\subsection{Some useful formulae related to interior and exterior angles.}

We first list some plausible facts from \cite{BDLR2017} that make
computations of the interior and exterior angles more amenable.

\begin{definition}\cite{BDLR2017}\label{alpha-angle}\label{beta-angle}
Let $P$ and $Q$ be two intermediate subfactors of a subfactor $N
\subset M$.  Then, the interior angle $\alpha^N_M(P,Q)$ between $P$
and $Q$ is given by
\begin{equation*}
 \alpha^N_M(P, Q) =  \cos^{-1} {\langle v_P,v_Q\rangle}_2,
 \end{equation*}
 where $v_P := \frac{e_P-e_1}{{\lVert e_P-e_1\rVert}_2}$, ${\langle x,
   y\rangle}_2 := \tr(y^*x)$ and ${\lVert x\rVert}_2 :=
 (\tr(x^*x))^{1/2}$.  And, the exterior angle between $P$ and $Q$ is
 given by $\beta^N_M(P, Q) = \alpha^M_{M_1}(P_1, Q_1)$.
\end{definition}

We will avoid being pedantic and often drop the superscript $N$ and
the subscript $M$ when the subfactor $N \subset M$ is clear from the
context. Recall that a (right) Pimsner-Popa basis for a subfactor $N
\subset M$ is a finite collection $\{\lambda_i : i \in I\}$ in $M$
satisfying $\lambda_i e_1 \lambda_i^* = 1$ or, equivalently, $x =
\sum_i E_N(x \lambda_i)\lambda_i^*$ for all $x \in M$ - see
\cite{PiPo, JS} for details.

\begin{theorem}\cite{BDLR2017}\label{Thm: alpha-beta}
For a quadruple $(N,P,Q,M)$, let $\tau_P=\tr(e_P)$,
$\tau_Q=\tr(e_Q)$  and $\tau = \tr(e_1)$. Then, the interior angle $\alpha (P,Q)$ satisfies
\begin{align}\label{Equ:alpha}
\cos \big(\alpha(P, Q)\big)
&=\frac{\tr(e_Pe_Q)-\tau}{\sqrt{\tau_P-\tau}\sqrt{\tau_Q-\tau}},
\end{align}
which, then, yields that
\begin{equation}\label{alpha-eqn-basis}
 \cos \big(\alpha (P, Q) \big) = \frac{\sum_{i,j} \tr_M\big( E^M_N
   (\lambda_i^* \mu_j) \mu_j^* \lambda_i\big) -1}{\sqrt{[P:N]
     -1}\sqrt{[Q:N] -1}}
\end{equation}
for any two Pimsner-Popa bases $\{\lambda_i\}$ and $\{\mu_j\}$ of $P/N$ and
$Q/N$, respectively. And, if the quadruple is extremal, i.e., $N
\subset M$ is extremal, then the exterior angle $\beta(P,Q)$ satisfies
\begin{align}\label{Equ:beta}
\cos\big( \beta(P, Q)\big) &=
\frac{\tr(e_Pe_Q)-\tau_P\tau_Q}{\sqrt{\tau_P-\tau_P^2}\sqrt{\tau_Q-\tau_Q^2}}.
\end{align}
\end{theorem}

The following useful expression for $\tr(e_P e_Q)$ is quite evident
from \Cref{Equ:alpha} and \Cref{alpha-eqn-basis}; the details can be
readily extracted from the proof of \cite[Proposition 2.14]{BDLR2017}.

\begin{lemma}\label{tr-epeq}
          Let $N \subset M$ be a subfactor  and $P$ and $Q$ be two intermediate
          subfactors. Then,
          \[
\tr_{M_1}(e_P e_Q) = \tau \sum_{i,j}\|E_N(\lambda_i^* \mu_j)\|_2^2
\]
for any two Pimsner-Popa bases $\{\lambda_i\}$ and $\{\mu_j\}$ of
$P/N$ and $Q/N$, respectively, where $\tau := [M:N]^{-1}$.
\end{lemma}

The following useful relationship between $\alpha(P,Q)$ and
$\beta(P,Q)$ was mentioned in \cite{BDLR2017}, following Definition
3.6, without any proof. For the sake of completeness, we include a
proof using the planar algebraic technique of Jones (though, only for the
extremal case, which will be enough for our requirements).
\begin{lemma}\label{angle-duality}
  For an extremal subfactor $N \subset M$ with intermediate subfactors
  $P$ and $Q$, we have
   \[
 \alpha^N_M (P, Q) =  \beta^M_{M_1} (P_1, Q_1).
\]
  \end{lemma}

\begin{proof}
  We have a tower
  \[
N \subset P \subset M  \subset P_1 \subset M_1  \subset P_2 \subset M_2  
\]
where $P_i \subset M_i \subset P_{i+1} = \la M_i, e_{P_i} \ra$ is a
basic construction with Jones projection $e_{P_i} = e_{0, i+1}:
L^2(M_i) \rar L^2(P_i)$, and $P_0:=P$ - see \cite[$\S$
  3]{BL}. Likewise, we have another tower
  \[
N \subset Q \subset M  \subset Q_1 = \la M, e_{Q}\ra \subset M_1  \subset Q_2 = \la M_1, {e_{Q_1}}\ra \subset M_2 . 
\]
From \cite[Lemma 4.2]{BL}, we have \psfrag{d}{$e_P$} $ e_{P_2} =
\vcenter{\psfrag{e}{$e_P$}\includegraphics[scale=0.4]{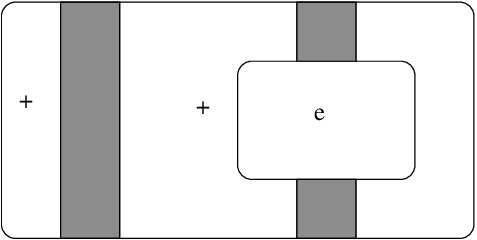}.}$ We
have a similar figure for $e_{Q_2}$ with respect to $e_Q$. From this
pictorial description, it is readily seen through pictures that
\[
\tr(e_P
e_Q) = \tr (e_{P_2} e_{Q_2}), \tr(e_P) = \tr(e_{P_2})\ \text{and}\ \tr (e_Q) =
\tr (e_{Q_2}).
\]
From \Cref{Equ:alpha}, we have
$
\cos \big(\alpha^N_M (P, Q) \big) = \frac{\tr(e_P e_Q) - \tau}{\sqrt{\tau_P - \tau} \sqrt{\tau_Q - \tau}}
$
and, similarly,
$
\cos \big(\alpha^{M_1}_{M_2} (P_2, Q_2) \big) = \frac{\tr(e_{P_2} e_{Q_2}) - \tau}{\sqrt{\tau_{P_2} - \tau} \sqrt{\tau_{Q_2} - \tau}}.
$
Finally, employing the above equalities obtained through pictures, we obtain
\[
\cos \big( \beta^M_{M_1}(P_1, Q_1)\big) = \cos \big(\alpha^{M_1}_{M_2}
(P_2, Q_2) \big) = \cos \big(\alpha^N_M (P, Q) \big),
\]
as was desired.  \end{proof}

\begin{figure}[h]
  \psfrag{N}{$N$} \psfrag{P}{$P$} \psfrag{M}{$M$}\psfrag{Q}{$Q$}
  \psfrag{P1}{$P_1$}\psfrag{Q1}{$Q_1$}\psfrag{M1}{$M_1$}
\psfrag{a}{$\alpha$}\psfrag{b}{$\beta$}
  \includegraphics[scale=0.3]{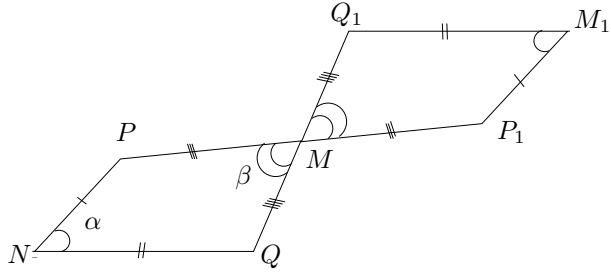}
\caption{Intuitive picture of a quadrilateral along with its dual}\label{fig-2}
\end{figure}

Recall that two subfactors $N \subset M$ and $
\mathcal{N}\subset\mathcal{M}$ are said to be isomorphic (denoted as
$(N\subset M)\cong (\mathcal{N}\subset \mathcal{M})$) if there exists
a $*$-isomorphism $\varphi$ from $M$ onto $\mathcal{M}$ such that
$\varphi(N)=\mathcal{N}.$ Likewise, two quadruples $(N, P, Q, M)$ and
$(\mathcal{N}, \mathcal{P}, \mathcal{Q}, \mathcal{M}) $ are said to be
isomorphic if there is an isomorphism $\varphi$ between the subfactors
$N \subset M$ and $ \mathcal{N}\subset\mathcal{M}$ such that $\varphi
(P) = \mathcal{P}$ and $\varphi (Q) = \mathcal{Q}$.

\begin{remark}\label{isomorphism}
Since Pimsner-Popa bases are preserved by isomorphisms of subfactors,
in view of \Cref{Thm: alpha-beta} and \Cref{tr-epeq}, we observe that an isomorphism
between two quadruples preserves interior and exterior angles.
\end{remark}

Recall that a quadruple $(N, P, Q, M)$ is said to be a commuting
square if $e_P e_Q = e_1 = e_Q e_P$. It is said to be a
cocommuting square if the dual quadruple $(M, P_1, Q_1, M)$ is a commuting
square. It is said to be non-degenerate (resp., irreducible) if
$\overline{\text{span}PQ} = M$ (resp., $N'\cap M= \C$). Further, it is
said to be a parallelogram if $\tau_P \tau_Q = \tau$ or, equivalently,
if $[M: P ] = [Q:N]$ or $[M:Q] = [P:N]$. And, a quadruple $(N,P,Q,M)$
is said to be a quadrilateral if $P\vee Q=M$ and $P\wedge Q=N.$

\begin{remark}\label{commuting-cocommuting}
  Commuting and cocommuting conditions have  very natural
interpretations in terms of above angles, viz., a quadruple $(N, P, Q,
M)$ is a commuting (resp., co-commuting) square if and only if
$\alpha(P,Q)$ (resp., $\beta(P,Q)$) equals $\pi/2$ - see \cite[$\S
  2$]{BDLR2017}.
\end{remark}

\subsection{Computation of angles for quadruples of crossed product algebras}\label{examples}

\begin{proposition}\label{ex-1}
Let $G$ be a finite group acting outerly on a $II_1$-factor $S$. Let
$H, K $ and $L$ be subgroups of $G$ such that $L \subseteq H \cap K$
and $H$ and $K$ are non-trivial.  Consider the quadruple $(N = S
\rtimes L, P = S \rtimes H, Q = S \rtimes K, M= S \rtimes G )$. Then,
\begin{equation}
  \tr(e_P e_Q)  = \frac{ |H \cap K|}{|G|},
  \end{equation}
\begin{equation}
  \cos\big( \alpha(P, Q)\big) =  \frac{|H \cap K| -|L|}{\sqrt{|H|-|L|} \sqrt{|K|-|L|}}
  \end{equation}
and
\begin{equation}
  \cos\big( \beta(P, Q)\big) =  \frac{\frac{|G|}{|HK|} - 1}{\sqrt{ [G:H]- 1} \sqrt{[G:K] -1}}.
\end{equation}
In particular, as is well known, $(N,P,Q,M)$ is a commuting  (resp., cocommuting) square if
and only if $H \cap K = L$ (resp., $G = HK$).
\end{proposition}

\begin{proof}
 Note that if $\alpha : G \rar \text{Aut}(S)$ denotes the action of
 $G$ on $S$, then there is a unitary representation $G \ni t \mapsto
 u_t \in B(L^2(S))$, such that $u_t(x\Omega) = \alpha_t(x) \Omega$ for
 all $x \in S$ - see \cite[$\S$ A.4]{JS}.

Fix left coset representatives $\{h_i : 1 \leq i \leq [H : L]\}$ and $
\{k_j : 1 \leq j \leq [K : L]\}$ of $L$ in $H$ and $K$,
respectively. Since $E_N(\sum_g x_g u_g) = \sum_l x_l u_l$, it follows
that $ \{h_i: 1 \leq i \leq [H:L]\}$ and $\{k_j : 1 \leq j \leq
[K:L]\}$ are (right) orthonormal bases for $P/N$ and $Q/N$,
respectively. So, by \Cref{tr-epeq}, we obtain
\begin{eqnarray*}
  \tr(e_P e_Q) & = & [G:L]^{-1}  \sum_{i,j}\tr_{M} \Big(
  E_N\big( u_{h_i}^{-1} u_{k_j}\big)u_{k_j}^{-1} u_{h_i} \Big) \\
  & = & [G:L]^{-1} \, \sum_{ \{ i,j : h_i^{-1}k_j \in L
    \}}\tr_{M} \Big( u_{h_i}^{-1} u_{k_j} u_{k_j}^{-1} u_{h_i} \Big)\quad \big(\text{since }
  E_N\Big(\sum_{g\in G}x_g u_g\Big) = \sum_{l\in L} x_l u_l \big)\\
  & = &[G:L]^{-1} |\{(i, j) : k_j \in h_iL\}|\\
  & = &[G:L]^{-1}\,|\{ (i, j) : h_iL \cap k_jL \neq
  \emptyset \}|;
\end{eqnarray*}
and note that the map \[
\{(i, j) : h_i L \cap k_j L \neq \emptyset\} \ni (i, j) \mapsto h_i L = k_j L \in (H \cap K)/L
\]
is a natural bijection; so that, $ \tr(e_P e_Q) = \frac{ |H \cap
  K|}{|G|}$.  Then, from \Cref{Equ:alpha}, we immediately obtain
\[
  \cos\big( \alpha(P, Q)\big)  =  \frac{|H \cap K| -|L|}{\sqrt{|H|-|L|} \sqrt{|K|-|L|}}
\]
and, from \Cref{Equ:beta}, through an elementary simplification, we deduce that
\[
\cos\big( \beta(P, Q)\big) =  \frac{\frac{|G|}{|HK|} - 1}{\sqrt{ [G:H]- 1} \sqrt{[G:K] -1}}.
\]
The commuting and cocommuting conditions follow from \Cref{commuting-cocommuting}.
\end{proof}

\begin{corollary}\label{fixed-angle}
Let $H, K, G$ and $S$ be as in \Cref{ex-1}.
Consider the quadruple $(N= S^G, P = S^H, Q = S^K, M = S)$. Then,
  \[
  \cos (\alpha (P, Q)) = \frac{ \frac{| G|}{|HK|}-1}{\sqrt{{[G:H]}-1}\sqrt{{[G:K]}-1}} 
  \]
  and
  \[
  \cos (\beta(P,Q)) = \frac{|H \cap K| - 1}{\sqrt{|H| -1}\sqrt{|K| -1}}.
\]
  In particular,  $(N, P , Q , M )$ is a
  commuting square if and only if $HK = G$. And, it is a cocommuting
  square if and only if $H \cap K = \{ e\}$.
  \end{corollary}
\begin{proof}
Since $S^G \subset S$ is extremal, by \Cref{angle-duality},   we have
   \[
 \alpha^N_M (P, Q) =  \beta^M_{M_1} (P_1, Q_1).
\]
Outhere, we have $M = S, P_1 = S \rtimes H$, $Q_1= S \rtimes K$ and
$M_1 = S \rtimes G$; so, by \Cref{ex-1} (taking $L$ to be the trivial
subgroup), we obtain
\[
 \cos \big(\alpha^N_M (S^H, S^K)\big) = \cos \big(\beta^M_{M_1} (S
 \rtimes H, S \rtimes K)\big) = \frac{ \frac{|
     G|}{|HK|}-1}{\sqrt{{[G:H]}-1}\sqrt{{[G:K]}-1}} .
 \]
 On the other hand, by definition, we have $\beta^N_M(P,Q) =
 \alpha^M_{M_1}(P_1, Q_1)$. Hence, by \Cref{ex-1}, we obtain
 \[
  \cos (\beta(P,Q)) = \frac{|H \cap K| - 1}{\sqrt{|H| -1}\sqrt{|K| -1}}.
 \]
 \end{proof}

It was shown in \cite[$\S$ 5]{BDLR2017} that the notion of
Sano-Watatani's set of angles does not agree with the notion of
interior angle.  Using \Cref{fixed-angle}, we add to that list and show
 that the Sano-Watatani's set of angles and the interior angle may not
be equal even if the former is a singleton.

\begin{example}
  Consider the quadruple $(N= R^G, P = R^H, Q = R^K, M = R)$ with the
  assumption that $H \cap K =\{e\}$, $|H\backslash G/H| = 2$, $H$ and
  $K$ are both non-trivial subgroups. Then, the  Sano-Watatani's set of angles $
  \mathrm{Ang}_M(P, Q)$ is a singleton and $ \{\alpha (P, Q)\} \neq
  \mathrm{Ang}_M(P, Q).$
  \end{example}

\begin{proof}
From \cite[Lemma 5.3 and Proposition 5.2]{SW}, we have
$\mathrm{Ang}_M(P, Q)$ is a singleton, namely,
\[
\mathrm{Ang}_M(P, Q) = \left\{\cos^{-1}
\left(\frac{|G| - |H||K|}{|K|(|G| -|H|)}\right)^{1/2}\right\}.
\]
And, by \Cref{fixed-angle}, we have
\[
\cos\big(\alpha(P, Q)\big) = \frac{ \frac{|
    G|}{|HK|}-1}{\sqrt{{[G:H]}-1}\sqrt{{[G:K]}-1}} = \frac{ \frac{|
    G|}{|H||K|}-1}{\sqrt{|G|/|H|-1}\sqrt{|G| |K|-1}},
\]
where the second equality follows because $H \cap K = \{e\}$ gives
$|HK| =|H||K|$.  Thus,
\(
\{\alpha (P, Q)\} =
\mathrm{Ang}_M(P, Q)
\)
if and only if
\begin{equation}\label{equality}
\left(\frac{|G| - |H||K|}{|K|(|G| -|H|)}\right)^{1/2} = \frac{ \frac{|
    G|}{|H||K|}-1}{\sqrt{|G|/|H|-1}\sqrt{|G| |K|-1}}.
\end{equation}
Note that $RHS = \left(\frac{|G| - |H||K|}{(|G| -|H|)}\right)^{1/2}
\cdot \left(\frac{|G| - |H||K|}{(|G| -|K|)
  |H||K|}\right)^{1/2}$. Hence (\ref{equality}) is true if and only if
\[
\frac{1}{|K|} = \frac{|G| - |H||K|}{(|G| -|K|)
  |H||K|},
\]
which is then true if and only if 
$|H| = 1$, which is not true since $H$ is not the trivial subgroup.
  \end{proof}

We conclude this subsection by deducing the following  well known fact.
\begin{example} 
Let $N \subset M$ be a subfactor and $G$ be a finite group acting
outerly (through $\alpha$) on $M$ . Then, $(N, N\rtimes G,M, M\rtimes G)$ is a commuting
square.
\end{example}

\begin{proof}
 Note that $E^{M\rtimes G}_N\big(\sum_g a_gu_g\big)= E^M_N(a_e).$
 Indeed, for any $b\in N$, we have
 \[
 \tr\Big(\big(\sum a_gu_g\big)b\Big)= \tr\Big(\sum a_g{\alpha}_g(b)u_g\Big)=\tr(a_eb),
 \] and, on
 the other hand, $\tr\big(E^M_N(a_e)b\big)=\tr\big(E^M_N(a_eb)\big)=\tr(a_eb).$

Let $\{\lambda_i\}$ be a (right) basis for $M/N$. Then, from
\Cref{Equ:alpha}, we obtain
 \[
 \cos\big(\alpha(N \rtimes G, M)\big)= \frac{\sum_{i,g}
   \tr\big(E_N({\lambda}^*_iu_g) u_g^*
   {\lambda}_i\big)-1}{\sqrt{[M:N]-1}\sqrt{|G|-1}}= \frac{\tr\big(\sum_i
   E_N({\lambda}^*_i){\lambda}_i\big)-1}{\sqrt{[M:N]-1}\sqrt{|G|-1}} = 0,
 \]
because $\sum_i \lambda_i E_N(\lambda_i^*) = 1$. Thus, $\alpha=
\pi/2$, and hence, by \Cref{commuting-cocommuting}, $(N, N \rtimes G,
M, M \rtimes G)$ is a commuting square.
\end{proof}

\section{Weyl group, Quadrilaterals and regularity}\label{quadrilaterals-regularity}

In this section we  focus on the analysis of irreducible
subfactors and quadrilaterals from the perspectives of Weyl group and
interior and exterior angles between intermediate subfactors.

 Let $N \subset M$ be a subfactor and let $\mathcal{U}(N)$ (resp.,
 $\mathcal{U}(M)$) denote the group of unitaries of $N$ (resp., $M$)
 and $\mathcal{N}_M(N) := \{ u \in \mathcal{U}(M): u N u^* = N\}$
 denote the group of unitary normalizers of $N$ in $M$. Clearly,
 $\mathcal{U}(N)$ is a normal subgroup of $\mathcal{N}_M(N)$. For a
 finite index subfactor $N \subset M$, one associates the so-called
 Weyl group, which we shall denote by $G$, defined as the quotient
 group $\mathcal{N}_M(N)/\mathcal{U}(N)$ (\cite{Hong, JoPo, Kos, Cho,
   PiPo2}). It is known that for an irreducible subfactor $N \subset
 M$, $G$ is a finite group with order less than or equal to $[M:N]$
 (see, for instance, \cite{Hong, PiPo2} as well as \cite{Kos}).

 \begin{example}\label{weyl-example}
Let $G$ be a finite group acting outerly on a $II_1$-factor $N$ and $H$
be a normal subgroup of $G$. Then, the Weyl group of the subfactor $N\rtimes H
\subset N \rtimes G$ is isomorphic to the quotient group $G/H$.
   \end{example}
 \begin{proof}
Fix a set of coset representatives $\{g_i: 1 \leq i \leq n=[G:H]\}$ of
$H$ in $G$. Then, $\{u_{g_i}: 1 \leq i \leq n\}$ forms a two sided
orthonormal basis for $(N \rtimes G )/N$ (where $u_g$'s are as in
\Cref{ex-1}). Clearly, the map $G/H \ni g_i H
\stackrel{\varphi}{\longmapsto} [u_{g_i}] \in G$ is a bijection. Then,
note that
\[
\varphi(g_i H)\varphi( g_j H) = [u_{g_i}][u_{g_j}] = [u_{g_i}u_{g_j}]= [u_{g_i g_j}]. 
\]
On the other hand, if $g_i g_j H = g_k H$, then $g_i g_k = g_k h$ for
some $h \in H$, which implies that $ u_{g_i g_j} = u_{g_k}u_h $, i.e.,
$[u_{g_i g_j} ] = [u_k] $ in $G$. Thus, $\varphi(g_i H\cdot g_j H) =
\varphi(g_k H) = [u_{g_k}]= [u_{g_i g_j} ] = \varphi(g_i H)\varphi(
g_j H) $ for all $1 \leq i, j \leq n$. Hence, $K/H \cong G$.
   \end{proof}

Now, we make some useful observations related to regularity and
orthonormal basis determinded by the Weyl group. Recall that a
subfactor $N \subset M$ is said to be regular if $\mathcal{N}_M(N)'' =
M$.

\begin{proposition}\label{intermediatebasis1}
Let $P$ be an intermediate $II_1$-factor of a subfactor $N \subset M$. Let
$e_P$ denote the canonical Jones projection for the basic construction
$P \subset M \subset P_1$ and $\{{\lambda}_i\}$ be a finite set in
$P$. Then, $\{{\lambda}_i\}$ is a Pimsner-Popa basis for $P/N$ if and
only if $\sum_i \lambda_i e_1{\lambda}^*_i=e_P$. 
\end{proposition}
\begin{proof}
If $\{{\lambda}_i\}$ is a Pimsner-Popa basis for $P/N$, then we know
that $\sum_i \lambda_i e_1 \lambda_i^* = e_P$ - see , for instance,
the proof of \cite[Proposition 2.14]{BDLR2017}. This proves necessity.

 To prove sufficiency, consider the basic construction $N\subset
 P\subset N_1$ with Jones projection $e^P_N$. Recall, from
 \cite{BL}, that this tower is isomorphic to the tower $Ne_P\subset
 e_PMe_P=Pe_P \subset e_PM_1e_P$ via a map $\phi: N_1 \rar e_P
 M_1 e_P$ satisfying $\phi(x)=xe_P$ for all $x\in P$.  The Jones
 projection for the second tower is given by $e_Pe_1=e_1$.  Note that
 $\sum_i {\lambda}_ie_Pe_1e_P{\lambda}^*_i= \sum_i
 {\lambda}_ie_1{\lambda}^*_i=e_P.$ Thus, we obtain
 $$
 \phi\big(\sum_i
 {\lambda}_ie^P_N{\lambda}^*_i\big)=\sum_i
 ({\lambda}_ie_P)e_1(e_P{\lambda}^*_i)=e_P=\phi(1).
 $$ This implies that $\sum \lambda_ie^P_N{\lambda}^*_i=1$ and, hence,
 $\{\lambda_i\}$ is a Pimsner-Popa basis for $P/N$.
\end{proof}

\begin{proposition}\label{p=ep}
Let $N\subset M$ be an irreducible subfactor and $P:=
\mathcal{N}_M(N)''$. If $\{u_g: g = [u_g] \in G\}$ denotes a
set of coset representatives of $G$ in $\mathcal{N}_M(N)$, then
$\{u_g: g \in G\} \subset \mathcal{N}_P(N)$ and it forms
  a two sided orthonormal basis for $P/N.$
\end{proposition}

\begin{proof}
Since $N \subset M$ is irreducible, it follows that $P$ is a
$II_1$-factor. By definition, we have $\mathcal{N}_P(N) \subset
\mathcal{N}_M(N)$. On the other hand, $\mathcal{N}_M(N) \subset
P$. So, if $u \in \mathcal{N}_M(N)$, then $u \in
\mathcal{N}_P(N)$. Hence, $\mathcal{N}_P(N) =
\mathcal{N}_M(N)$. Therefore, we conclude that $N \subset P$ is a
regular subfactor and also that the Weyl group of $N \subset P$ is the
same as that of $N \subset M$.
  
Then, since $N \subset P$ is regular and irreducible, we conclude, from
\cite[Lemma 3.1]{Hong}, that $\{u_g:g\in G\}$ forms a two sided
orthonormal basis for $P/N$.
\end{proof}

Above two propositions yield the following improvement of \cite[Lemma 3.1]{Hong}:
\begin{theorem}\label{reg-onb}
 Let $N\subset M$ be an irreducible subfactor and 
 $\{u_g : g \in G\}$ be a set of coset representatives of $G$ in
 $\mathcal{N}_M(N)$. Then, the following are equivalent:
 \begin{enumerate}
\item $[M:N] =|G|$.
\item $\{ u_g : g \in G\}$ is a two sided  orthonormal basis for $M/N$.
 \item $N \subset M$ is regular.
\end{enumerate}
\end{theorem}

\begin{proof}
$(1) \Leftrightarrow (2):$ Let $p :=\sum_g u_g e_1 u_g^*$. Then, $\tr(p) = \sum_{g \in G} \tr(u_g
  e_1 u_g^*) = |G| [M:N]^{-1} $. Thus, $\{ u_g : g \in G\}$ is an
  orthonormal basis for $M/N$ if and only if $[M:N] = |G|$. 

 $(2) \Leftrightarrow (3):$ This equivalence follows immediately from 
  \Cref{intermediatebasis1} and \Cref{p=ep}.
\end{proof}

Analogous to the well known Goldman's Theorem (\cite{Gol}, also see,
\cite{Jon}) for a subfactor with index $2$, it is known that an
irreducible regular subfactor $N \subset M$ can be realized as the
group subfactor $N \subset N \rtimes G$, where $G$ is the Weyl group
of $N \subset M$ which acts outerly on $N$ - see \cite{Hong, PiPo2,
  Kos} and the references therein. As a consequence, we deduce the
following version of Goldman's type theorem for irreducible
quadrilaterals.

 \begin{theorem}\label{hk-generate-g}
Let $(N, P, Q, M)$ be an irreducible quadrilateral such that $N\subset
P$ and $N \subset Q$ are both regular. Then, $G$ acts outerly on $N$
and $(N, P, Q, M) = (N, N \rtimes H, N \rtimes K, N \rtimes G)$, where
$H, K$ and $G$ are the Weyl groups of $N \subset P$, $N \subset Q$ and
$N \subset M$, respectively.
\end{theorem}
\begin{proof}
 First, note that $N \subset M$ is regular because \[ M= P \vee Q =
 \mathcal{N}_P(N)'' \vee \mathcal{N}_Q(N)'' \subseteq
 \{\mathcal{N}_P(N) \cup \mathcal{N}_Q(N)\}'' \subseteq
 \mathcal{N}_M(N)'' \subseteq M.
 \]
Then, since $N \subset M$ is regular and irreducible, Hong \cite{Hong}
had shown that if $N_{-1} \subset N \subset M$ is an instance of
downward basic construction with Jones projection $e_{-1}$, then there
is a representation $G\ni g \mapsto v_g \in \mathcal{U}(N_{-1}'\cap
M)$ such that $v_g \in \mathcal{N}_M(N)$ for all $g\in G$, $M = \{N,
v_g: g \in G\}''$ and $ G \in g \mapsto Ad_{v_g}\in Aut(N)$ is an
outer action of $G$ on $N$, i.e., $(N \subset M) = (N \subset N
\rtimes G)$ - see \cite[Lemma 3.3 and Theorem 3.1]{Hong}. Also, for
each $g \in G$, the coset $ v_g\, \mathcal{U}(N) = g$ in $G$.

  So, by Galois correspondence, $P = N \rtimes H'$ and $Q = N \rtimes K'$
  for unique subgroups $H'$ and $K'$ of $G$. We assert that  $H = H'$ and $K = K'$.

We have $P = N \rtimes H' = \{N, v_{h'}: h' \in H'\}''$. So, for each
$h' \in H'$, $v_{h'} \in P$; thus, $\{v_{h'}: h' \in H'\} \subset
\mathcal{N}_P(N)$. Also, $h'= [v_{h'}]\in
\mathcal{N}_P(N)/\mathcal{U}(N) = H $, so that $H' \subset H$. As seen
in \Cref{weyl-example}, $H' \cong H$; so, we must have $|H'| = |H|$ and hence
$H' = H$. Likewise, we obtain $K = K'$. Hence, $(N, P, Q, M) = (N, N
\rtimes H, N \rtimes K, N \rtimes G)$. 
\end{proof}

\begin{corollary}\label{hk-generate-g2}
 Let $(N, P, Q, M)$ be an irreducible quadrilateral such that $N\subset P$
and $N \subset Q$ are both regular. Then, the Weyl groups of $N \subset P$ and $N \subset Q$ together
generate the Weyl group of $N \subset M$.
\end{corollary}
\begin{proof}
 Let $G'$ be the subgroup of $G$ generated by $H$ and $K$, then $N
  \rtimes G' \subseteq N \rtimes G$. Also, since $M = P\vee Q$, we have
  \[
  N \rtimes G = (N \rtimes H)\vee (N \rtimes K) \subseteq N \rtimes G'\subseteq N \rtimes G.
  \]
  Hence, by Galois correspondence again, we must have $G = G'$, i.e.,
  $G$ is generated by its subgroups $H$ and $K$.
\end{proof}
We have the following  partial converse of \Cref{hk-generate-g2}.
\begin{proposition}
 Let $(N,P,Q,M)$ be an irreducible quadruple such that $N\subset P$
 and $N\subset Q$ are both regular. If $N\subset M$ is regular
 and the Weyl groups of $N\subset P$ and $N\subset Q$ together
 generate the Weyl group of $N\subset M$, then $M=P\vee Q$.
\end{proposition}
\begin{proof}
 Fix any set of coset
  representatives $\{u_g: g \in G\}$ of $G$ in
  $\mathcal{N}_M(N)$. Since $N \subset M$ is regular, $\{u_g: g \in
  G\}$ forms a two sided orthonormal basis for $M/N$, by
  \Cref{reg-onb}. Note that each $g$ in $G$ is a word in $H \cup K$
  and for any pair $g, g' \in G$, we have $[u_{gg'}] = gg' =
  [u_g][u_{g'}]= [u_g u_{g'}]$, so that $u_{gg'} = v u_g u_g' $ for some $v\in
  \mathcal{U}(N)$. Thus, $M = \sum_g N u_g \subseteq (\sum_h N u_h)
  \vee (\sum_k u_k) = P \vee Q$.
\end{proof}

Following corollary first appeared implicitly in the proof of
\cite[Theorem 6.2]{SW}. We include it here, as an application of
\Cref{hk-generate-g} and \Cref{hk-generate-g2}.

\begin{corollary}
Let $(N, P, Q, M)$ be an irreducible quadrilateral with $[P:N] = 2 =
[Q:N]$. Then, $[M:N]$ is an even integer and the Weyl group of $N \subset
M$ is isomorphic to the Dihedral group of order $2n$, where $n = [M:P]=[M:Q]$.
\end{corollary}
\begin{proof}
  By Goldman's Theorem (\cite{Gol, Jon}), we know that $(N \subset
  P)\cong (N \subset N \rtimes_{\sigma} \Z_2)$ and $ (N \subset
  Q)\cong (N \subset N \rtimes_{\tau} \Z_2)$ for some outer actions
  $\sigma$ and $ \tau$ of $\Z_2$ on $N$ and hence both $N \subset P$
  and $N \subset Q$ are regular. Then, by \Cref{reg-onb}, we obtain
  $|H| = [P:N] =2$ and $|K| =[Q:N] =2$, where $H$ and $K$ are as in
  \Cref{hk-generate-g}. So, $H$ and $K$ are both cyclic of order
  $2$. By \Cref{hk-generate-g}, $N \subset M$ is regular and hence
  $[M:N] = |G|$ by \Cref{reg-onb}. Also, by \Cref{hk-generate-g}, $G$
  is a finite group generated by $H$ and $K$. Thus, $G$ is generated
  by two elements which are both of order $2$.  Hence, by
  \cite[Theorem 6.8]{Suz}, $G$ is isomorphic to the Dihedral group of
  order $2n$.
  \end{proof}
We conclude this section with the demonstration of a direct
relationship between angles and Weyl groups of interemdiate subfactors
of an irreducible quadruple. As above, for a quadruple $(N,P,Q,M)$, we
denote by $G,H$ and $K$ the Weyl groups of $N\subset M, N\subset P$
and $N\subset Q$, respectively. First, we deduce the relationship for
an irreducible quadrilateral.

\begin{theorem}\label{betacomutationforregularquadrilateral}
Let $(N,P,Q,M)$ be an irreducible quadrilateral such that $N\subseteq
 P$ and $N\subseteq Q$ are both regular. Then, $\alpha(P,Q)=\pi/2$,
 i.e., $(N, P, Q, M)$ is a commuting square, and
 $$\cos \big(\beta(P,Q)\big)=\displaystyle \frac{\frac{\lvert G \rvert}{\lvert
     H\rvert \lvert K\rvert}-1}{\sqrt{[G:H]-1}\sqrt{[G:K]-1}}.$$

 In particular, $(N, P, Q, M)$ is a cocommuting square if and only if $G = HK$.
\end{theorem}
\begin{proof}
From \Cref{hk-generate-g}, we have $(N, P, Q, M) = (N , N \rtimes H, N
\rtimes K, N \rtimes G)$. Since $N = P \wedge Q$, $H \cap K$ must be
trivial because $N \subseteq N \rtimes (H \cap K) = P \cap Q = N$. The
expressions for $\alpha$ and $\beta$ then follow from \Cref{ex-1} and
the fact that $|HK| = |H||K|$ when $H \cap K$ is trivial.
  \end{proof}

More generally, we have the following relationship.
\begin{theorem}\label{main1}
Let $(N,P,Q,M)$ be an irreducible quadruple such that $N\subseteq P$
and $N\subseteq Q$ are both regular.  
Then,
\[
  \cos\big( \alpha(P,Q)\big)= \frac{|H \cap K| - 1}{\sqrt{|H| -1}\sqrt{|K| -1}
 }\] and
 \[\cos \big( \beta(P,Q)\big) = \frac{\frac{[M:N]}{|HK|} -1}{\sqrt{[M:P]
     -1}\sqrt{[M:Q] -1}}.
\]
 In particular, $(N, P, Q, M)$ is a commuting square if and only if
 $H \cap K$ is trivial if and only  $P \cap Q = N$. And, $(N, P, Q,
 M)$ is a cocommuting square if and only if $|G| = |HK| =
 [M:N]$. \end{theorem}

\begin{proof}
  Since $N \subset M$ is irreducible, $P \vee Q$ is a $II_1$-factor.
  Consider the irreducible quadruple $(N, P, Q, P\vee Q)$. Then, by
  \Cref{hk-generate-g}, $(N, P, Q, P\vee Q) = (N, N \rtimes H, N
  \rtimes K, N \rtimes G')$ where $G'$ is the Weyl group of $N \subset
  P \vee Q$. Hence, by \Cref{ex-1}, we obtain 
\[
  \cos\big( \alpha^N_{P \vee Q}(P,Q)\big)= \frac{|H \cap K| -
    1}{\sqrt{|H| -1}\sqrt{|K| -1} }.
  \]
  And, it is known that $\alpha^N_{P \vee Q}(P,Q) = \alpha^N_{M}(P,Q)$
  - see \cite[Proposition 2.16]{BDLR2017}. And, since $P \cap Q = N
  \rtimes (H \cap K)$, $(N, P, Q, M)$ is a commuting square if and
  only $H\cap K$ is trivial, by \Cref{commuting-cocommuting}. Also, $H
  \cap K $ is trivial if and only if $P \cap Q = N$.
  
  On the other hand, being irreducible, $N \subset M$ is
   extremal. So, by \Cref{Equ:beta}, the exterior angle
   between $P$ and $Q$ is given by
   \begin{eqnarray*}
\cos \big(\beta(P, Q)\big)  & = & \frac{\tr(e_P e_Q) -\tau_P \tau_Q}{\sqrt{\tau_P - \tau_p^2} \sqrt{\tau_Q - \tau_Q^2}}
\\
& = & \frac{[M:N]^{-1} |H \cap K| - [M:P]^{-1}[M:Q]^{-1}}{\sqrt{[M:P]^{-1} - [M:P]^{-2}}\sqrt{[M:Q]^{-1} - [M:Q]^{-2}}}
\\
& = & \frac{|H \cap K| [P:N]^{-1} [M:Q] -1 }{\sqrt{[M:P] - 1}\sqrt{[M:Q] - 1}}
\\
&     = &  \frac{\frac{[M:N]}{|HK|} -1}{\sqrt{[M:P] -1}\sqrt{[M:Q] -1}},
 \end{eqnarray*}
where we have used the equalities $\tr(e_P e_Q) = \tau \sum_{h \in H, k
  \in K} \| E_N(u_h^* u_k)\|_2^2 = [M:N]^{-1}|H \cap K|$ by
\Cref{tr-epeq} and the well known formula $|HK| = |H| |K||H \cap
K|^{-1}$.

By \Cref{commuting-cocommuting} again, $(N, P, Q, M)$ is a cocommuting
square if and only if $\beta (P,Q) = \pi/2$, i.e., if and only if
$|HK| = [M:N]$. Note that $|G| \leq [M:N]$ (see \cite{Hong, PiPo2})
and $HK \subseteq G$. So, $|HK| = [M:N]$ if and only if $|G| = |HK|=
[M:N]$.
  \end{proof}

\begin{remark}
 Sano and Watatani (\cite[Theorem 6.1]{SW}) had proved that an
 irreducible quadrilateral $(N,P,Q,M)$ with $[P:N] = 2 =[Q:N]$ is
 always a commuting square. Thus, Theorem
 \ref{betacomutationforregularquadrilateral} can also be thought of as a
 generalization of that result.
\end{remark}

\section{Possible values of interior and exterior angles}\label{values}
It is a very natural curiousity to know the possible values of interior
and exterior angles between intermediate subfactor. As a first attempt
in this direction, we make some calculations in the irreducible set
up.

Prior to that we recall two auxiliary positive operators associated to
a quadruple (from \cite{BDLR2017, SW}) whose norms are equal, and show
that this common entity has a direct relationship with the possible
values of interior and exterior angles.

\subsection{Two auxiliary operators associated to a qudruple}\label{auxiliary} Consider a quadruple $(N,P,Q,M)$. Let $\{\lambda_i:i\in I\}$ and
$\{\mu_j:j\in J\}$ be (right) Pimsner-Popa bases for $P/N$ and $Q/N$,
respectively. Consider two positive operators $p(P,Q)$ and $p(Q,P)$
given by
$$
p(P,Q)= \sum_{i,j}{\lambda_i}\mu_j e_1 {\mu}^*_j{\lambda}^*_i\quad \text{and}\quad
p(Q,P)= \sum_{i,j}\mu_j \lambda_i e_1 {\lambda}^*_i {\mu}^*_j.$$

\begin{remark} By \cite[Lemma 2.18]{BDLR2017}, $p(P,Q)$ and $p(Q,P)$ are both
independent of choices of bases. And, by \cite[Proposition
  2.22]{BDLR2017}, $Jp(P,Q)J = p(Q,P)$, where $J$ is the usual modular
conjugation operator on $L^2(M)$; so that, $\|p(P,Q)\| =
\|p(Q,P)\|$.
\end{remark}

\begin{notation} For a quadruple $(N, P, Q, M)$, let 
 $r:=\frac{[P:N]}{[M:Q]}=\frac{[Q:N]}{[M:P]}$ and
  $\lambda:=\lVert p(P,Q)\rVert = \|p(Q,P)\|$. 
\end{notation}

Recall that for a self adjoint element $x$ in a von Neumann algebra
$\mathcal{M}$, its support is given by $s(x):= \inf\{ p \in
\mathcal{P}(\mathcal{M}): px = x =xp\}$. We will need the following
useful lemma which follows from \cite[Proposition 2.25 $\&$
  Lemma 3.2]{BDLR2017}.

\begin{lemma}\cite{BDLR2017}\label{sp-lemma}
  If $(N,P,Q,M)$ is a quadruple such that $N \subset M$ is
  irreducible, then $ \lambda = [Q:N]\, \tr(p(P,
  Q) e_Q) $ and $s(p(P,Q)) = p(P,Q)/ \lambda$. In particular, $\tr
  (s(p(P,Q))) = r/\lambda$ and $p(P,Q)$ is a projection if and only if
  $\lambda = 1$.
  \end{lemma}
It turns out that $s(p(P,Q))$ is a minimal projection in $P'\cap Q_1$
which is central as well.
 \begin{proposition}
 Let $(N,P,Q,M)$ be quadruple such that $N\subset M$
 irreducible. Then, $\frac {p(P,Q)}{\lambda}$ (resp.,
 $\frac {p(Q,P)}{\lambda}$) is a minimal  projection in
 $P^{\prime}\cap Q_1$ (resp., $Q^{\prime}\cap P_1$) which is also central.
\end{proposition}
\begin{proof}
  By \Cref{sp-lemma}, $\frac{1}{\lambda}p(P,Q)$ is a
  projection. Further, by \cite[Proposition 2.25]{BDLR2017}, we have
  $p(P,Q)= [P:N]E^{N^{\prime}}_{P^{\prime}}(e_Q)\in P^{\prime}\cap
  Q_1$.  We first show that $\frac{1}{\lambda} p(P,Q)$ is
  minimal in $P'\cap Q_1$. Consider any projection $q \in P^{\prime}\cap Q_1$
  satisfying $0\leq q \leq\frac{1}{\lambda} p(P,Q)$. Then,
  $q=\frac{q}{\lambda}p(P,Q).$ We also have $qe_Q= [M:Q] E_M^{Q_1}(q
  e_Q)e_Q$ (by the Pushdown Lemma \cite[Lemma 1.2]{PiPo}).  Clearly,
  $E_M(qe_Q)\in N^{\prime}\cap M$. Thus, irreduciblility of
  $N\subseteq M$ implies that $qe_Q=te_Q$ for the scalar $t = [M:Q]
  E_M^{Q_1}(q e_Q) $.  Therefore,
\begin{eqnarray*}
 q
 & = & \frac{q}{\lambda} p(P,Q)\\
 & = & \frac{q}{\lambda} [P:N] E^{N^{\prime}}_{P^{\prime}}(e_Q)\\
 & = & \frac{[P:N]}{\lambda} E^{N^{\prime}}_{P^{\prime}}(qe_Q)\\
 & = & \frac{[P:N]}{\lambda}tE^{N^{\prime}}_{P^{\prime}}(e_Q)\\
 & = & \frac{t}{\lambda} p(P,Q).
\end{eqnarray*}
Since $q$ and $\frac{1}{\lambda} p(P,Q)$ are projections we conclude
that $t^2=t$. Therefore, $q=0$ or $p(P,Q)$. Since $q$ was arbitrary,
this proves the minimality of $\frac{p(P,Q)}{\lambda}.$

We now prove that $\frac{1}{\lambda}p(P,Q)$ is a central projection in
$P^{\prime}\cap Q_1.$ For this, we first show that $e_Q$ is a minimal
central projection in $N^{\prime}\cap Q_1$.  Let $u$ be an arbitrary
unitary in $N^{\prime}\cap Q_1$. Then, by the Pushdown Lemma again, we
have $ue_Q=[M:Q]E_M(ue_Q)e_Q$. But clearly $E_M(ue_Q)\in
N^{\prime}\cap M=\C$. Thus, $ue_Qu^*=te_Q$ for some scalar $t$.  Since
$te_Q$ is a non-zero projection, we must have $t = 1$; so that, $ue_Q =
e_Q u$ for all $u \in \mathcal{U}(N'\cap Q_1)$, thereby implying that
$e_Q$ is central in $N^{\prime}\cap Q_1.$ This shows that
\[
 vp(P,Q)v^* = [P:N] E^{N^{\prime}}_{P^{\prime}}(ve_Qv^*) = 
 [P:N]E^{N^{\prime}}_{P^{\prime}}(e_Q) = p(P,Q)
\]
for all $v \in \mathcal{U}(P'\cap Q_1)$ and we are done.

Assertion about $p(Q,P)$ then follows from the fact that $p(Q,P) = J
p(P,Q) J$ - see \cite[ Proposition 2.22]{BDLR2017}.
\end{proof}

As asserted above, we now present the direct relationship that exists
between the values of the interior and exterior angles and the common norm of the
above two auxiliary operators.
 \begin{proposition}\label{formulaalphabeta}
Let $(N,P,Q,M)$ be a finite index quadruple of $II_1$-factors with $N\subset M$ irreducible. Then,
\begin{equation}\label{alpha-lambda}
\cos(\alpha(P,Q))=
\displaystyle\frac{(\lambda-1)}{\sqrt{[P:N]-1}\sqrt{[Q:N]-1}}\end{equation}
and \begin{equation}\label{beta-lambda-r} \cos(\beta(P,Q)) =
  \displaystyle\frac{(\lambda-r)}{\sqrt{[P:N]-r}\sqrt{[Q:N]-r}}. \end{equation}
\end{proposition}
\begin{proof}
From \Cref{Equ:alpha}, we have
\[
 \cos(\alpha(P,Q))= \frac{\tr(e_P
   e_Q)-\tau}{\sqrt{\tr(e_P)-\tau}\sqrt{\tr(e_Q)-\tau}}=\frac{[M:N]\tr(e_Pe_Q)-1}{\sqrt{[P:N]-1}\sqrt{[Q:N]-1}}.
 \]
 It can be shown that $p(P,Q) =[Q:N]E^{M_1}_{Q_1}(e_P)$ - see the proof of
 \cite [Proposition 2.25]{BDLR2017}; hence, $p(P,Q)
 e_Q=[Q:N]E^{M_1}_{Q_1}(e_Pe_Q).$ Thus,
 $\tr(p(P,Q)e_Q)=[Q:N]\tr(e_Pe_Q)$. This, along with the fact that
 $p(P,Q) e_Q=\lambda e_Q$ (because $N \subset M$ is irreducible - see
 the proof of \cite[Lemma 3.2]{BDLR2017}), yields
 $\tr(e_Pe_Q)=\frac{\lambda \tr(e_Q)}{[Q:N]} = \lambda \tau,$ that is
 $[M:N]\tr(e_Pe_Q)=\lambda$. Thus, we obtain
\[
\cos(\alpha(P,Q))=
 \frac{\lambda-1}{\sqrt{[P:N]-1}\sqrt{[Q:N]-1}}.
 \]
 On the other hand, being irreducible, $N
 \subset M$ is extremal. So, from \Cref{Equ:beta}, we have
 \begin{eqnarray*}
   \cos(\beta(P,Q)) & = & \frac{\tr(e_P e_Q) - \tau_P
     \tau_Q}{\sqrt{\tau_P - \tau_P^2}\sqrt{\tau_Q - \tau_Q^2}}\\ & = &
   \frac{\tau \lambda - \tau_P \tau_Q}{\sqrt{\tau_P -
       \tau_P^2}\sqrt{\tau_Q - \tau_Q^2}}\\ & = &
   \frac{\lambda-r}{\sqrt{[P:N]-\frac{[P:N]}{[M:P]}}\sqrt{[Q:N]-\frac{[Q:N]}{[M:Q]}}}.
   \end{eqnarray*}
 Then, note that
 \begin{align*}
 &
   \left([P:N]-\frac{[P:N]}{[M:P]}\right)\left([Q:N]-\frac{[Q:N]}{[M:Q]}\right)
   \\ & \qquad=
      [P:N][Q:N]-\frac{[P:N][Q:N]}{[M:Q]}-\frac{[P:N][Q:N]}{[M:P]}+\frac{[P:N][Q:N]}{[M:P][M:Q]}\\ &\qquad=
      [P:N][Q:N]-r[Q:N]-r[P:N]+r^2\\ & \qquad=([Q:N]-r)([P:N]-r),
\end{align*}
and we are done. \end{proof}

\begin{proposition}
Let $(N,P,Q,M)$ be a commuting square with $N\subset M$ irreducible. Then, 
$$\cos(\beta(P,Q))= \displaystyle
\frac{r^{-1}-1}{\sqrt{[M:P]-1}\sqrt{[M:Q]-1}}.$$ And, if $(N,P,Q,M)$
is a a cocommuting square with $N\subset M$ irreducible, then
 $$\cos(\alpha(P,Q))=\displaystyle \frac{r-1}{\sqrt{[P:N]-1}\sqrt{[Q:N]-1}}.$$
\end{proposition}
\begin{proof}
 The formula for $\beta(P,Q)$ is easy and is left to the
 reader. Cocommuting square implies $\beta(P,Q)=\pi/2.$ Thus,
 $\tr(e_Pe_Q)=\tau_P\tau_Q$. Now simply use
 $\frac{\tau_P\tau_Q}{\tau}=r$ and the formula follows from the
 definition of $\alpha(P,Q)$.
\end{proof}

\subsection{Values of angles in the irreducible setup}

In order to determine the values of interior and exterior angles
between intermeidate subfactors, as is evident from
\Cref{formulaalphabeta}, it becomes important to know the possible
values of $r$ and $\lambda$. Recall, from \cite{pop}, Popa's set of
relative-dimensions of projections in $M$ relative to $N$ given by
 $$
 \Lambda(M,N)=\{\alpha\in \mathbb{R}: \exists~~
\text{a projection}~~ f_0 \in M ~~\text{such that}~~ E_N(f_0)=\alpha
1_N\}.
$$

\begin{lemma}\label{rbylambda}
  For an irreducible quadruple $(N, P, Q, M)$,
  $\tr\big(s\big(p(P,Q)\big)\big)=\frac{r}{\lambda}\in \Lambda(M_1,M).$
  Also, $(1-\frac{r}{\lambda})\in \Lambda(M_1,M)$.
\end{lemma}
\begin{proof}
 Let $\{\lambda_i\}$ be a basis for $P/N$. Then, $p(P,Q) : = \sum
 \lambda_i e_Q{\lambda_i}^* \in M_1$ and clearly \(
 E^{M_1}_M\big(p(P,Q)\big)=\sum
 \lambda_iE^{M_1}_M(e_Q){\lambda}^*_i=\displaystyle\frac{\sum
   \lambda_i{\lambda}^*_i}{[M:Q]}=\frac{[P:N]}{[M:Q]}=r.  \) Thus,
 $E^{M_1}_M(\frac{1}{\lambda}p(P,Q))=\frac{r}{\lambda}$. And, by \Cref{sp-lemma}, the operator $\frac{1}{\lambda}p(P,Q) =
 s(p(P,Q)) $ is a projection. This completes the proof.
\end{proof}

For a commuting square $(N,P,Q,M)$ with $N\subset M$ irreducible, it
is known that $\lambda = 1$ (see \cite[Proposition
  2.20]{BDLR2017}). Thus, we deduce the following:
\begin{corollary}
 If $(N,P,Q,M)$ is a commuting square with $N\subset M$ irreducible, then $r\in \Lambda(M_1,M)$.
 And, if $(N,P,Q,M)$ is a parallelogram, then ${\lambda}^{-1}\in
 \Lambda(M_1,M).$
\end{corollary}

 Consider the polynomials $P_n(x)$ for $n \geq 0$ (introduced by
 Jones in \cite{Jon} and) defined recursively by $P_{0}=1, P_1=1,
 P_{n+1}(x)=P_n(x)-xP_{n-1}(x), n>0.$ Thus, $P_2(x)=1-x, P_3(x)=1-2x$
 and so on.  From \cite{Jon}, we know that $P_k\big(
 \frac{1}{4{\cos}^2\pi/(n+2)}\big)>0$ for $0\leq k\leq n-1$, and
 $P_n\big(\frac{1}{4{\cos}^2\pi/(n+2)}\big)=0.$ Furthermore,
 $P_k(\epsilon)>0$ for all $\epsilon \leq 1/4$ and $k\geq 0$. Also, by
 definition, we have $\frac{\tau
   P_{n-1}(\tau)}{P_n(\tau)}=1-\frac{P_{n+1}(\tau)}{P_n(\tau)}$ for
 all $n\geq 1$, where, as is standard, $\tau:=[M:N]^{-1}$.

While trying to determine the possible entries of the set $\Lambda(M,
N)$, Popa  \cite{pop} proved the following  theorem:

 \begin{theorem}\cite{pop}
  \label{popa}
  Let $N\subset M$ be a subfactor  of finite index.
  \begin{enumerate}
   \item If  $[M:N]=4 {\cos}^2\big(\frac{\pi}{n+2}\big)$ for some $n\geq 1$, then
   $$\Lambda(M,N)=\big\{0\big\}\cup \bigg\{\frac{\tau P_{k-1}(\tau)}{P_k(\tau)}:0\leq k\leq n-1\bigg\}=\bigg\{\frac{P_k(\tau)}{P_{k-1}(\tau)}:0\leq k\leq n\bigg\}.$$
\item If $[M:N]\geq 4$ and $t\leq 1/2$ is so that $t(1-t)=\tau$, then
$$ \Lambda(M,N)\cap (0,t) = \bigg\{\frac{\tau P_{k-1}(\tau)}{P_k(\tau)}:k\geq 0\bigg\}.$$
  \end{enumerate}
\end{theorem}

Since a subfactor with index less than $4$ does not admit any
intermediate subfactor, for any non-trivial quadrilateral $(N, P, Q,
M)$, we always have $[M:N] \geq 4$.
 \begin{proposition}\label{rlambdacomputation}
 Let $(N,P,Q,M)$ be an irreducible quadruple and let $0<t\leq 1/2$ be
 such that $t(1-t)=\tau$. Then, either $\frac{r}{\lambda}\geq t$ or
 $\frac{r}{\lambda}=\frac{\tau P_{k-1}(\tau)}{P_{k}(\tau)}$ for some
 $k\geq 0$.
 \end{proposition}
 
\begin{proof}
 This follows from \Cref{popa} and
 \Cref{rbylambda}.
\end{proof}

We may thus compute the interior and exterior angles in this specific
situation, as follows.
\begin{theorem}\label{valuesofangle}
Let $(N,P,Q,M)$ be a  quadruple with $N\subset M$ irreducible and let $0<t\leq 1/2$ be
such that $t(1-t)=\tau$. If $r/\lambda \geq t$, then,
$$\cos(\alpha(P,Q)) \leq
\frac{[P:N][Q:N](1-t)-1}{\sqrt{[P:N]-1}\sqrt{[Q:N]-1}}$$ and $$
\cos(\beta(P,Q))\leq \displaystyle
\frac{\frac{1}{t}-1}{\sqrt{[M:P]-1}\sqrt{[M:Q]-1}}.$$
And, if $r/\lambda < t$, then, 
$$ \cos(\alpha(P,Q)) =
\displaystyle\frac{[P:N][Q:N]\frac{P_k(\tau)}{P_{k-1}(\tau)}-1}{\sqrt{[P:N]-1}\sqrt{[Q:N]-1}}
$$
and
$$\cos(\beta(P,Q))=\displaystyle \frac{\frac{P_k(\tau)}{\tau
    P_{k-1}(\tau)}-1}{\sqrt{[M:P]-1}\sqrt{[M:Q]-1}}$$ for some $k\geq
0.$
\end{theorem}
\begin{proof}
 First, suppose that $r/\lambda\geq t$.  Thus, $\lambda\leq
 r/t$. Observe that $r/t=[P:N][Q:N](1-t)$; so, from
 \Cref{alpha-lambda}, we obtain the first inequality. Also,
 $\lambda-r \leq r(1/t-1)$. Thus, from \Cref{beta-lambda-r}, we
 obtain $$\cos(\beta(P,Q))\leq \displaystyle
 \frac{r(1/t-1)}{\sqrt{[P:N]-r}\sqrt{[Q:N]-r}}=\displaystyle
 \frac{(1/t-1)}{\sqrt{[M:Q]-1}\sqrt{[M:P]-1}}.$$
  
 Next, suppose that $N \subset M$ is irreducible and $r/\lambda <t$. By
 \Cref{rlambdacomputation}, we have
 $\lambda=\frac{r}{\tau}\frac{P_k(\tau)}{P_{k-1}(\tau)}$, for some
 $k\geq 0.$ Observe that $r/\tau=[P:N][Q:N].$ Thus, by
 \Cref{alpha-lambda}, we obtain $$ \cos(\alpha(P,Q)) =
 \displaystyle\frac{[P:N][Q:N]\frac{P_k(\tau)}{P_{k-1}(\tau)}-1}{\sqrt{[P:N]-1}\sqrt{[Q:N]-1}}
$$ and, by \Cref{beta-lambda-r}, we obtain $$\cos(\beta(P,Q))=
\displaystyle \frac{r\bigg(\frac{P_k(\tau)}{\tau
    P_{k-1}(\tau)}-1\bigg)}{\sqrt{[P:N]-r}\sqrt{[Q:N]-r}}=\displaystyle
\frac{\frac{P_k(\tau)}{\tau
    P_{k-1}(\tau)}-1}{\sqrt{[M:P]-1}\sqrt{[M:Q]-1}}.$$
 \end{proof}

\begin{theorem}\label{valuesofbeta}
 Let $(N,P,Q,M)$ be an irreducible quadrilateral such that $N\subset
 P$ and $N\subset Q$ are both regular and suppose $[P:N]=2$.  Then,
 $\cos(\beta(P,Q))= \displaystyle
 \frac{P_2(m/2)}{\sqrt{P_2(\delta^2/2)P_3(m/2)}}$, where $m = [M:Q]
 \in \N$ and, as usual, $\delta:=\sqrt{[M:N]}.$
  \end{theorem}
\begin{proof}
Let $H, K$ and $G$ denote the Weyl groups of $N \subset P$, $N \subset
Q$ and $N \subset M$, respectively. We have $m = [M:Q] = \frac{
  [M:N]}{[Q:N]} = \frac{|G|}{|K|}$ (by
\Cref{hk-generate-g} and \Cref{reg-onb}) and $\lvert H\rvert=2$. Thus,
by \Cref{betacomutationforregularquadrilateral}, we obtain

\begin{eqnarray*}
\cos(\beta(P,Q)) & = &  \frac{\frac{\lvert G\rvert}{\lvert
    H\rvert \lvert K\rvert}-1}{\sqrt{\frac{\vert G\rvert}{\lvert
      H\rvert}-1}\sqrt{\frac{\lvert G \rvert}{\lvert K\rvert}-1}}\\
 & = & \displaystyle
  \frac{\frac{m}{2}-1}{\sqrt{(\delta^2/2)-1}\sqrt{m-1}}\\  &
   = &
  \frac{P_2(m/2)}{\sqrt{P_2(\delta^2/2)}\sqrt{P_3(m/2)}},
\end{eqnarray*}
as was desired.
 \end{proof}

\begin{corollary}\label{valuesofbeta2}
 Let $(N,P,Q,M)$ be an irreducible quadrilateral such that $[P:N] = 2
 = [Q:N]$. Then, $\alpha(P,Q)=\pi/2$ and
 $\beta(P,Q)={\cos}^{-1}\left(\frac{P_2(m/2)}{P_3(m/2)}\right)$, where
 $m = [M:P]=[M:Q]$ is an integer. In particular, $\beta(P,Q)>\pi/3.$
\end{corollary}

\begin{proof}
 This follows from  \Cref{valuesofbeta} and the fact that
 $P_2(\delta^2/2)=P_2(m)=P_3(m/2)$, where $[M:P]=[M:Q]=m$.
\end{proof}

\begin{corollary}
 \label{valuesofbeta3}
  Let $(N,P,Q,M)$ be an irreducible quadrilateral such that
  $N\subseteq P$ and $N\subseteq Q$ are both regular with $[P:N]=2$
  and $[Q:N]=3.$ Then, $\alpha(P,Q)=\pi/2$ and $\beta(P,Q)>
  {\cos}^{-1}\bigg(\frac{1}{\sqrt{6}}\bigg).$
\end{corollary}
\begin{proof}
By \Cref{betacomutationforregularquadrilateral}, we have $\alpha(P,Q)
= \pi/2$ and taking $m = [G:K]$, we
obtain $${\cos}^2(\beta(P,Q))=\displaystyle
\frac{m^2/4-m+1}{3m^2/2-5m/2+1}< 1/6,$$ where the inequality follows
from a routine comparison using the fact that $m \geq 2$.
\end{proof}

\section{Certain bounds on angles and their implications}\label{angle-bounds}
In this section, we observe that when one leg of an irreducible
quadruple is assumed to be regular, it enforces certain
bounds on interior and exterior angles, which then imposes some bounds
on the index of the other leg. 
\begin{theorem}\label{inequality}
 Let $(N,P,Q,M)$ be an irreducible quadruple such that
 $N\subset P$ is regular. Then,
 $$\cos\big(\alpha(P,Q)\big)\leq \Bigg(\sqrt{\frac{[P:N]-1}{[Q:N]-1}}\Bigg)$$ and
 $$\cos(\beta(P,Q))\leq \Bigg(\sqrt{\frac{[P:N]-r}{[Q:N]-r}}\Bigg).$$
\end{theorem}
\begin{proof}
  For any $u\in \mathcal{N}_M(P)$ we have $uE_Q(u^*)\in N^{\prime}\cap
 M =\C$. To see this, let $n\in N$ be arbitrary. Then,
 $uE_Q(u^*).n=uE_Q(u^*n)=uE_Q(u^*nu.u^*)=n.uE_Q(u^*).$

 Let $uE_Q(u^*)=t \in \C.$ If $t\neq 0$ we get $E_Q(u^*) = t u^*$ so
 that $u\in Q$, which implies that $t= uE_Q(u^*) = u u^* = 1$. Thus,
 $uE_Q(u^*) \in \{0, 1\}$ for all $u \in \mathcal{N}_M(P)$.

 Using \Cref{reg-onb}, fix an orthonormal basis $\{u_i\} \subset \mathcal{N}_P(N)$ for
 $P/N$.  Consider the auxiliary operator $p(P,Q) =\sum_iu_ie_Qu^*_i$.
 Since $N \subset M$ is irreducible, we have $p(P,Q) e_Q = \lambda
 e_Q$ (see \cite[Lemma 3.2]{BDLR2017}), where $\lambda = \|p(P,Q)\|$.
 Also, $p(P,Q)e_Q=\sum_iu_ie_Qu^*_ie_Q=\sum_iu_iE_Q(u^*_i)e_Q\in \C
 e_Q$, which yields $ \sum_i u_i E_Q(u_i^*) = \lambda$.  In
 particular, we obtain $ 0 \leq \lambda\leq \lvert
 \mathcal{N}_P(N)/\mathcal{U}(N)\rvert = [P:N]$. Thus, $\lambda-1 \leq
         [P:N]-1$ and the lower bound for $\alpha$ follows from
         \Cref{formulaalphabeta}(1).\smallskip

We also have $\lambda-r\leq [P:N]-r.$ So the lower bound for $\beta$  follows from 
 \Cref{formulaalphabeta}.
\end{proof}
\begin{remark}
Note that in above proof, we also observed that $\|p(P,Q)\|$ is an
integer less than or equal to $[P:N]$.
  \end{remark}

Above theorem imposes some immediate bounds on $[Q:N]$ in terms of $[P:N]$,
as follows.
\begin{corollary}
  Let $(N, P, Q, M)$ be as in \Cref{inequality}. Then, we have the following:
  \begin{enumerate}
\item If $\alpha(P,Q)  \leq \pi/3$, then $[Q:N] \leq 4[P:N] -3$.
  \item If $\alpha(P,Q)  \leq \pi/4$, then $[Q:N] \leq 2[P:N] -1$.
    \item If $\alpha(P,Q)  \leq \pi/6$, then $[Q:N] \leq 4/3 [P:N] -1/3$.
    \end{enumerate}
\end{corollary}
\begin{proof}
If $\alpha (P, Q) \leq \pi/3$, then $\cos \alpha (P,Q) \geq \cos
(\pi/3) = 1/2$. So, $\frac{[P:N] - 1}{[Q:N] - 1} \geq 1/4$ and hence $[Q:N] \leq 4 [P:N]
-3$. Others follow similarly.
\end{proof}

As a consequence, when we intuitively try to visualize an irreducible
quadruple $(N, P, Q, M)$ with $[P:N] = 2$ as a $4$-sided structure in
plane (as in \Cref{fig-2}), then it seems that the smaller is the interior angle between $P$ and
$Q$ the shorter is the length (or index) of
$N \subset Q$. This assertion is supported by the following observations:

If $[Q:N] > 7/3$, then it follows from above corollary that
$\alpha(P,Q) > \pi/6$.  Likewise, If $[Q:N] > 3$, then $\alpha(P,Q) >
\pi/4$.  And, if $[Q:N] > 5$, then $\alpha(P,Q) > \pi/3$. In
particular, since $P$ is a minimal subfactor, i.e., $N \subset
P$ admits no intermediate subfactor, in the last scenario, $Q$ cannot
be a minimal subfactor because, by \cite{BDLR2017}, we know that the interior
angle between two minimal intermedite subfactors is always less than $\pi/3$.
Also, observe that if $\alpha(P,Q) \leq \pi/4$, then $[Q:N] \leq 3$
and hence $N \subset Q$ must be a Jones' subfactor.

The reader can get a better feeling of above assertion by making similar calculations for
an irreducible quadruple $(N, P, Q, M)$ such that $P \subset N$ is
regular with $[P:N] = n$, for arbitrary $n$.

\end{document}